\documentclass[12pt]{article}

\usepackage{amssymb}
\usepackage{amsmath}
\usepackage{amsthm}
\usepackage{anysize}
\usepackage{dsfont}

\usepackage{enumitem}

\marginsize{1.25in}{1.25in}{0.8in}{1.5in}

\newtheorem{theorem}{Theorem}[section]

\newtheorem{corollary}[theorem]{Corollary}
\newtheorem{definition}[theorem]{Definition}

\DeclareMathOperator{\tr}{tr}

\usepackage{hyperref}
\hypersetup{
colorlinks=true,
allcolors=black
}

\title{Bounds via spectral radius-preserving row sum expansions}
\author{Joseph P. Stover\footnote{Department of Mathematics, Gonzaga University, 502 E Boone Ave, Spokane, WA 99258. E-mail:  \href{mailto:stover@gonzaga.edu}{stover@gonzaga.edu}}}

\begin{document}

\maketitle

\let\thefootnote\relax\footnotetext{Keywords: block matrix; spectral radius; nonnegative matrix}
\let\thefootnote\relax\footnotetext{AMS 2020 Subject Classification: 15A18, 15A42}

%
\begin{abstract}\addcontentsline{toc}{section}{Abstract}
We show a simple method for constructing larger matrices but preserving the spectral radius. This yields a sufficient criteria for two square matrices of arbitrary dimension have the same spectral radius, a way to compare spectral radii of two matrices, and a way to derive new upper and lower bounds on spectral radius which give the standard row sum bounds as a special case.
\end{abstract}

\section{Introduction}
For non-negative square matrices it is well known that the spectral radius is between the minimum and maximum row sums. If the row rums are constant, then that sum is the spectral radius. It is also know that equality between the spectral radius and either the minimal or maximal row sum for an irreducible matrix implies its row sums are constant. See Theorem 1.1 in \cite{nonnegmat_minc}. Here we will take a square matrix and increase its dimension by expanding one of its diagonal elements (or a square diagonal block) into a larger square block while still preserving the spectral radius of the original matrix. Then we use this to derive new bounds on the spectral radius and show the standard row sum bounds are a special case.

Let $M\geq0$ be a nonnegative square matrix with spectral radius $\rho(M)$. Consider diagonal element $M_{ii}=d$. Via a permutation we can simply assume $d$ is the lower right most diagonal component
$$M=\left[\begin{matrix}
A & \mathbf b\\
\mathbf c^\top & d
\end{matrix}\right]$$
where $A$ is an $(n-1)\times (n-1)$ block and $\mathbf b$ and $\mathbf c$ are $(n-1)\times 1$ vectors.

We now create a new matrix
$$\tilde M=\left[\begin{matrix}
A & B\\
C & D
\end{matrix}\right].$$
where $D$ is a square matrix of arbitrary dimension $s\times s$ with constant row sums equal to $d$, $B$ is a matrix whose $i^{th}$ row sum equals the $i^{th}$ element of $\mathbf b$ (but arbitrary entries otherwise), and $C$ is a matrix with each row identical to $\mathbf c^\top$. We refer to $\tilde M$ as a \textit{row sum expansion of $M$ on $d$}. The overal dimensions of $\tilde M$ are now $(n-1+s)\times(n-1+s)$.

It is in fact the case that $\rho(M)=\rho(\tilde M)$ (which we prove in Theorem \ref{thm:row_contr_spec_rad_eq} below). Similarly, one can perform \textit{column sum expansions} instead (with the rules on $B$ and $C$ symmetrically adjusted). One can also expand multiple diagonal elements simultaneously, some using row sums and some using column sums, as long as the adjacent blocks are created appropriately. In fact, one can perform any sequence of permutations, transposes, and various mixes of column and row sum expansions (always transforming diagonal elements into square blocks) while still preserving the spectral radius. The resulting matrix (with larger dimension) always has the same spectral radius.

We first give some examples of row and column sum expansions to familiarize the reader with the process. Then in Section \ref{sec:specradpres} we prove that this process preserves the spectral radius. Finally in Section \ref{sec:specradbound} we give bounds on the spectral radius of a matrix based on this idea and give a way to compare spectral radii of two matrices with different dimensions.

\paragraph{Examples of matrix expansions}
Let $A$ and $B$ be nonnegative square matrices, we write $A\sim B$ if and only if $\rho(A)=\rho(B)$.

When we perform an expansion of a diagonal element into a block, we sometimes add additional parentheses to illustrate the new block when was previously a single component. Here is an introductory example:
$$
[5]\sim
\left[\begin{matrix}
2 & 3\\
4 & 1
\end{matrix}\right]\sim
\left[\begin{matrix}
2 & 4\\
3 & 1
\end{matrix}\right]\sim
\left[\!\!\begin{array}{ccc}
\left(\begin{array}{cc}
	1 & 1 \\
	2 & 0 
\end{array}\right) & 
\left(\begin{array}{c}
	4 \\
	4 
\end{array}\right)\\
\left(\begin{array}{cc}
	1 & 2  
\end{array}\right) & 1
\end{array}\!\!\right].
$$
where the first step we simply created a $2\times2$ matrix with constant row sums of 5 and then transposed the result in the next step. The final step, we row-sum expanded the first diagonal element. This requires us to repeat the horizontal off-diagonal component 4 (for a row sum expansion we always replicate exactly any rows to the left or right) and we can expand the vertical off-diagonal element 3 into an arbitrary $1\times2$ block as long as its row sum is 3 (for a row sum expansion we are allowed to replicate any above or below matrix elements into arbitrary rows that each sum to the original component they are expanded from).

Here is an example with a simultaneous row sum expansion on both diagonal elements into different sized diagonal blocks. In this example, we simply require the row sums in each block to be constant and equal to the original matrix component being expanded. 
$$
\left[\begin{matrix}
5 & 7\\
2 & 4
\end{matrix}\right]\sim
\left[\!\!\begin{array}{ccc}
\left(\begin{array}{cc}
	3 & 2 \\
	0 & 5 
\end{array}\right) & 
\left(\begin{array}{ccc}
	1 & 4 & 2\\
	3 & 1.5 & 2.5
\end{array}\right)\\
\left(\begin{array}{cc}
	1 & 1  \\
	0 & 2  \\
	1.5 & 0.5  
\end{array}\right) & 
\left(\begin{array}{ccc}
	3 & 0.5 & 0.5\\
	1 & 1 & 2\\
	0 & 4 & 0
\end{array}\right)
\end{array}\!\!\right].
$$

Here is an example where we perform a column expansion on the first diagonal element and a row expansion on the second. In this case, the lower left block must be the lower left matrix component repeated in all positions, but the upper right block just needs to sum to the original upper right component:
$$
\left[\begin{matrix}
5 & 8\\
7 & 3
\end{matrix}\right]\sim
\left[\!\!\begin{array}{ccc}
\left(\begin{array}{cc}
	4 & 2 \\
	1 & 3 
\end{array}\right) & 
\!\!\!\left(\begin{array}{ccc}
	0 & 1 \\
	2 & 5 
\end{array}\right)\\
\left(\begin{array}{cc}
	7 & 7  \\
	7 & 7    
\end{array}\right) & 
\!\!\!\left(\begin{array}{ccc}
	3 & 0 \\
	2 & 1 
\end{array}\right)
\end{array}\!\!\right].
$$
Again, we could go on concocting complicated mixes of transposes, permutations, and row and column sum expansions. This allows the overall mix of row and column sums to become almost any collection of nonnegative numbers, thus the standard bounds on spectral radius do not make it obvious that the spectral radius is preserved.

\section{Row sum expansions preserve spectral radius}\label{sec:specradpres}
We now prove that row sum expansions preserve the spectral radius. Here it is important that we are expanding diagonal elements into square blocks. In general, a row expansion that does not expand diagonal elements into square blocks does not preserve the spectral radius.
\begin{theorem}
	\label{thm:row_contr_spec_rad_eq}
Let $M$ be a non-negative matrix of dimension $n\times n$, and let $\tilde M$ be a row sum expansion of $M$ on any one of its diagonal elements. Then $\rho(M)=\rho(\tilde M)$.
\end{theorem}
\begin{proof}
Without loss of generality, we prove the theorem for the last diagonal component of $M$ by permutational similarity. Let
$$M=\left[\begin{matrix}
	A & \mathbf b\\
	\mathbf c^\top & d
\end{matrix}\right]$$	
where $A$ is a $(n-1)\times (n-1)$ block, $\mathbf b$ and $\mathbf c$ is a $(n-1)\times 1$ vectors, and $d$ is a single diagonal component. 

Since $M\geq0$, we know that its spectral radius $\rho\geq0$ is an eigenvalue with nonnegative, nonzero left eigenvector $(\mathbf x^\top,y)$ where $\mathbf x$ has $n-1$ components and $y$ is a single component (see Theorem 8.3.1 in \cite{matan_horn}). Assume the eigenvector sums to one $y+\sum_{i=1}^{n-1}x_i=1$. Note that we do not require any uniqueness, so there may be several such eigenvectors.

We have that 
$$\begin{aligned}
	(\mathbf x^\top, y)\left[\begin{matrix}
		A & \mathbf b\\
		\mathbf c^\top & d
	\end{matrix}\right]&=(\mathbf x^\top A+y\mathbf c^\top, \mathbf x^\top\mathbf b+yd)\\
	&=(\rho \mathbf x^\top , \rho y)
\end{aligned}$$
thus giving 
\begin{equation}
	\label{eq:M_eigs1}
	\mathbf x^\top\mathbf b=y(\rho-d)
\end{equation}
\begin{equation}
	\label{eq:M_eigs2}
	\mathbf x^\top (\rho I_{n-1}-A)=y\mathbf c^\top
\end{equation}
It is already known that $\rho\geq d$. If $\rho=d$, then $\mathbf x$ is orthogonal to $\mathbf b$ and thus has zeros wherever $\mathbf b$ has positive components (recall that everything is nonnegative). If $\rho$ is also an eigenvalue of principle submatrix $A$, then we can find infinitely many $\mathbf x$, but, again, we do not require uniqueness here. We just require some nonzero, nonnegative eigenvector to exist.

We now create a new $(n-1+s)\times (n-1+s)$ matrix $\tilde M\geq0$ by expanding $d$ into a $s\times s$ matrix $D$ with constant row sums $\sum_{j=1}^sD_{ij}=d$ for all $i=1,2,\ldots,s$.
$$\tilde M=\left[\begin{matrix}
	A & B\\
	C & D
\end{matrix}\right]$$
where $B$ is an $(n-1)\times s$ block whose row sums equal the elements of vector $\mathbf b$: $\sum_{j=1}^sB_{ij}=b_i$ for $i=1,2,\ldots,n-1$, and $C$ is an $s\times (n-1)$ block where each row is identical $C_{i-}=\mathbf c^\top$ for each $i=1,2,\ldots,s$. 

Similarly, this expanded matrix has a nonnegative spectral radius $\tilde \rho\geq 0$ which is an eigenvalue with nonnegative, nonzero eigenvector $(\tilde {\mathbf x}^\top,\mathbf y^\top)$. Again we assume that $\sum_{i=1}^s y_i+\sum_{i=1}^{n-1}x_i=1$. We will show that $\tilde {\mathbf x}=\mathbf x$ and $\sum y_i=y$ and $\tilde\rho=\rho$.

First, we show that $\rho,({\mathbf x}^\top,\mathbf y^\top)$ is an eigenpair for $\tilde M$ with $\mathbf y$ being a vector that solves	
\begin{equation}\label{eq:rhoy}
	{\mathbf x}^\top B=\mathbf y^\top(\rho I_s-D).
\end{equation}

There may be several such vectors $\mathbf y$ (if $\rho=d$ the constant row sum for matrix $D$). We choose a $\mathbf y$ so that this eigenvector sums to one requiring $y=\sum_{i=1}^s y_i=1-\sum_{i=1}^{n-1}x_i$. This is possible since the row sums of $\mathbf x^\top B$ are all equal to $\mathbf x^\top \mathbf b$ and the row sums of $\mathbf y^\top(\rho I_s-D)$ are all equal to $\left(\sum_{i=1}^s y_i\right)(\rho-d)=y(\rho-d)$ which are equivalent by (\ref{eq:M_eigs1}).

Since block $C$ simply has vector $\mathbf c^\top$ repeated $s$ times for its rows, we have 

\begin{equation}\label{eq:yC}
	\mathbf y^\top C=\left(\sum_{i=1}^s y_i\right) \mathbf c^\top=y\mathbf c^\top.
\end{equation}

Using (\ref{eq:M_eigs2}) and (\ref{eq:yC}) we see that 
\begin{equation}\label{eq:xAyC}
	{\mathbf x}^\top A+\mathbf y^\top C={\mathbf x}^\top A+y\mathbf c^\top=\rho {\mathbf x}^\top.
\end{equation}

Using (\ref{eq:rhoy}), we see that 
\begin{equation}\label{eq:xByD}
	{\mathbf x}^\top B+\mathbf y^\top D=\rho \mathbf y^\top.
\end{equation}

Now we have by using (\ref{eq:xAyC}) and (\ref{eq:xByD}) that 
$$\begin{aligned}
({\mathbf x}^\top, \mathbf y^\top)\left[\begin{matrix}
		A & B\\
		C & D
	\end{matrix}\right]&=({\mathbf x}^\top A+\mathbf y^\top C, {\mathbf x}^\top B+\mathbf y^\top D).\\
&=( \rho {\mathbf x}^\top , {\rho} \mathbf{y}^\top)
\end{aligned}$$
This shows that $\rho,({\mathbf x}^\top,\mathbf y^\top)$ is an eigenpair for $\tilde M$. This establishes that $\tilde \rho \geq \rho$.

Now we show that $\tilde\rho=\rho$ by contradiction. Assume that $\tilde\rho>\rho$ and that we have eigenpair $\tilde\rho,(\tilde{\mathbf x}^\top,\tilde{\mathbf y}^\top)$ for $\tilde M$. Then we have
$$\tilde{\mathbf x}^\top B=\tilde{\mathbf y}^\top(\tilde\rho I_s-D)$$
which, after checking the row sums on each side, yields
\begin{equation}\label{eq:tyrhod}
	\tilde{\mathbf x}^\top \mathbf b=\left(\sum_{i=1}^s \tilde y_i\right)(\tilde\rho-d).
\end{equation}
And of course, we also have that
\begin{equation}\label{eq:tyC}
	\tilde{\mathbf x}^\top (\tilde\rho I_{n-1}-A)=\tilde{\mathbf y}^\top C=\left(\sum_{i=1}^s \tilde y_i\right) \mathbf c^\top
\end{equation}
Together, (\ref{eq:tyrhod}) and (\ref{eq:tyC}) and using (\ref{eq:M_eigs1}) and (\ref{eq:M_eigs2}) with $\tilde\rho,\tilde{\mathbf x},\tilde y= \sum_{i=1}^s \tilde y_i$ instead of $\rho,\mathbf x,y$ show that $\tilde\rho,(\tilde{\mathbf x}^\top,\tilde y)$ is an eigenpair for $M$ which contradicts the fact that $\rho$ is its spectral radius. Hence $\tilde\rho=\rho$.

\end{proof}

This theorem applies to column sum expansions as well (since this is equivalent to row sum expansions on $A^\top$), and hence it applies be any sequence of row and column sum expansions, permutations, and transposes since each step preserves the spectral radius.

\section{Spectral radius bounds via contractions}\label{sec:specradbound}

If we were to perform the reverse procedure, contracting our matrix to lower its dimension, then unless the row or column sums in all blocks are constant (maybe this is achievable with some permutation or transpose), we would have to break the preservation of spectral radius. However, we use this to derive bounds on the spectral radius.

Again, letting $M$ be an $n\times n$ nonnegative square matrix, we define upward and downward row and column sum contractions.

\begin{definition}
	\label{def:drscontr}
	Let $M$ be a nonnegative square matrix. A \emph{downward row sum contraction} of $M$ is denoted $M^\downarrow$ and is a matrix created as follows. First we partition $M$ into a block matrix with $k^2$ blocks denoted $M=(A_{ij})$ for $i,j$ in $\{1,2,...,k\}$ 
	requiring that all diagonal blocks $A_{ii}$ are square matrices. 
	In each block, we take the minimum row sum: if block $A_{ij}$ has size $s\times t$ then $a_{ij}=\min_u \{\sum_{v=1}^t (A_{ij})_{uv}\}$. Then we have $M^\downarrow_{ij}=a_{ij}$.
\end{definition}
Similarly we define \emph{upward row sum contraction} of $M$ denoted by $M^\uparrow$ which uses the maximum row sum in each block: $M^\uparrow=(b_{ij})$ where $b_{ij}=\max_u \{\sum_{v=1}^t (B_{ij})_{uv}\}$. Naturally, we also define upward and downward column sum contractions (which are identical to applying row sum contractions to $M^\top$). We could also perform any sequence of permutations, transposes, and row and column sum contractions to arrive at a new matrix of lower dimension.

\paragraph{An example}
Consider the matrix $M$ partitioned as follows:
$$M=\left[\!
\begin{array}{c|ccc|cc}
	1 & 2 & 0 & 3 & 7 & 1\\\hline
	4 & 1 & 0 & 5 & 1 & 0\\
	0 & 2 & 3 & 0 & 0 & 5\\
	5 & 2 & 0 & 1 & 3 & 1\\\hline
	0 & 0 & 3 & 5 & 0 & 2\\
	4 & 1 & 3 & 0 & 1 & 3
\end{array}\!\right].
$$
Then taking the minimum and maximum row sum in each block gives
$$
M^\downarrow=\left[\!
\begin{array}{ccc}
	1 & 5 & 8\\
	0 & 3 & 1 \\
	0 & 4 & 2 
\end{array}\!\right],
\quad
M^\uparrow=\left[\!
\begin{array}{ccc}
	1 & 5 & 8\\
	5 & 6 & 5 \\
	4 & 8 & 4 
\end{array}\!\right].
$$
Notice that before reducing the dimensions of the matrix, we can increase or decrease components of $M$ in order to create an ordering. For this example, here is one possibility with the changed components in bold:
$$
\tilde M^\downarrow=\left[\!
\begin{array}{c|ccc|cc}
	1 & 2 & 0 & 3 & 7 & 1\\\hline
	\mathbf 0 & 1 & 0 & \mathbf 2 & 1 & 0\\
	0 & 2 & \mathbf 1 & 0 & 0 & \mathbf 1\\
	\mathbf 0 & 2 & 0 & 1 & \mathbf 0 & 1\\\hline
	0 & 0 & 3 & \mathbf 1 & 0 & 2\\
	\mathbf 0 & 1 & 3 & 0 & 1 & \mathbf 1
\end{array}\!\right]
\leq \left[\!
\begin{array}{c|ccc|cc}
	1 & 2 & 0 & 3 & 7 & 1\\\hline
	4 & 1 & 0 & 5 & 1 & 0\\
	0 & 2 & 3 & 0 & 0 & 5\\
	5 & 2 & 0 & 1 & 3 & 1\\\hline
	0 & 0 & 3 & 5 & 0 & 2\\
	4 & 1 & 3 & 0 & 1 & 3
\end{array}\!\right]
\leq 
\left[\!
\begin{array}{c|ccc|cc}
	1 & 2 & 0 & 3 & 7 & 1\\\hline
	\mathbf 5 & 1 & 0 & 5 & \mathbf 5 & 0\\
	\mathbf 5 & 2 & 3 & \mathbf 1 & 0 & 5\\
	5 & 2 & \mathbf 3 & 1 & 3 & \mathbf 2\\\hline
	\mathbf 4 & 0 & 3 & 5 & 0 & \mathbf 4\\
	4 & 1 & 3 & \mathbf 4 & 1 & 3
\end{array}\!\right]=\tilde M^\uparrow
$$
It should be intuitively clear that this will always be possible, to arbitrarily decrease or increase some components to make all row sums in each block constant and equal to the minimum or maximum row sum. It is important that we keep the diagonal blocks as square principle submatrices in order to use Theorem \ref{thm:row_contr_spec_rad_eq}. As mentioned above already, a row expansion that does not expand diagonal elements into square blocks does not generally preserve the spectral radius. Hence if our diagonal block of constant sum are now square, then reducing the dimension can change the spectral radius, which we wish to avoid here.

It is a standard result for ordered square nonnegative matrices that $A\leq B$ (component-wise: $A_{ij}\leq B_{ij}$ for all $i,j$) that the spectral radii are also ordered $\rho(A)\leq\rho(B)$ (see Corollary 8.1.20 in \cite{matan_horn}). 
Thus we get 
$$
\rho(\tilde M^\downarrow)\leq \rho(M)\leq\rho(\tilde M^\uparrow)
$$
for any arbitrary downward and upward adjustment of $M$ to make constant row sums in each block. We formalize this as a theorem in the general case.

Let $C_s^\downarrow(M)$ be the set of all $s\times s$ downward contractions of matrix $M$ and $C_s^\uparrow(M)$ be the set of all $s\times s$ upward contractions. Each $M^\downarrow \in C_s^\downarrow(M)$ is an $s\times s$ matrix whose components are equal to the minimum row sum of each block for some partition of $M$, and similarly for any $M^\uparrow\in C_s^\uparrow(M)$. Note that these are distinct from $\tilde M^\downarrow$ and $\tilde M^\uparrow$ which have the same dimensions as $M$ but with some components adjusted downward or upward, respectively, but resulting in the ordering: $\tilde M^\downarrow \leq M \leq \tilde M^\uparrow$. By Theorem \ref{thm:row_contr_spec_rad_eq} we have that $\rho(M^\downarrow)=\rho(\tilde M^\downarrow)$ even though they have different sizes and also that $\rho(M^\uparrow)=\rho(\tilde M^\uparrow)$.

Note that trivially, $[\min_i \{\sum_{j=1}^n M_{ij}\}]\in C_1^\downarrow(M)$ and $[\max_i \{\sum_{j=1}^n M_{ij}\}]\in C_1^\uparrow(M)$ for the minimal and maximal row sums of $M$ viewed as $1\times1$ matrices and that $M\in C_n^\downarrow(M)\cap C_n^\uparrow(M)$.

Let $C^\downarrow(M)=\cup_{j=1}^{n-1} C_j^\downarrow(M)$ and $C^\uparrow(M)=\cup_{j=1}^{n-1} C_j^\uparrow(M)$ be the sets of all possible downward and upward contractions of $M$ (with dimensions less than $M$). It should be clear that each of these sets is finite. The process of creating a downward or upward contraction involves some finite sequence of permutations, transposes, and dimension decreases. Eventually, including more permutations or transposes will create no new matrices, and obviously the dimension can only be decreased to $1\time 1$ at the lowest. Transpose and permutations also preserve the spectral radius, so it is only the dimension reductions that can make the spectral radius smaller or larger. Note that a matrix in $C^\downarrow(M)$ may include several dimension decrease steps in its creation, which could mean the spectral radius undergoes several decreases relative to $M$, and similar for $C^\uparrow(M)$.

\begin{theorem}
\label{thm:rhoboundscontr}
	Let $C^\downarrow(M)$ and $C^\uparrow(M)$ be the set of all downward and upward contractions of nonnegative square matrix $M$. Then
$$\max_{M^\downarrow\in C^\downarrow(M)} \{\rho(M^\downarrow)\}\leq \rho(M)\leq \min_{M^\uparrow\in C^\uparrow(M)} \{\rho(M^\uparrow)\}.$$
\end{theorem}
\begin{proof}
	For any $M^\downarrow\in C^\downarrow(M)$, we can create a matrix $\tilde M^\downarrow$ by adjusting individual components of $M$ downward and have $\rho(M^\downarrow)=\rho(M^\downarrow)$ by Theorem \ref{thm:row_contr_spec_rad_eq}. We also have that $\tilde M^\downarrow \leq M$, and similarly for any $M^\uparrow\in C^\uparrow(M)$. So that we always have $\rho(M^\downarrow)=\rho(\tilde M^\downarrow)\leq \rho(M)\leq \rho(\tilde M^\uparrow)=\rho(M^\uparrow)$. The result follows by taking the maximum on the left and minimum on the right.
\end{proof}

The standard bounds on $\rho(M)$ using its row sums is a special case of Theorem \ref{thm:row_contr_spec_rad_eq}.
\begin{corollary}\label{cor:stdrowsumbounds}
	We have that $[\min_i \{\sum_{j=1}^n M_{ij}\}]\in C_1^\downarrow(M)$ and $[\max_i \{\sum_{j=1}^n M_{ij}\}]\in C_1^\uparrow(M)$ which gives $\min_i \{\sum_{j=1}^n M_{ij}\}\leq \rho(M) \leq \max_i \{\sum_{j=1}^n M_{ij}\}$.
\end{corollary}
There is no guarantee that Theorem \ref{thm:rhoboundscontr} gives better bounds on $\rho(M)$ than the standard row sum bounds, but it does improve the bounds in at least some cases (some examples are given below).

Here is another direct corollary which just considers particular $2\times 2$ row sum contractions, for which calculating the spectral radius is straightforward.
\begin{corollary}\label{cor:rho2x2contr}
	Let $A=(a_{ij})$ be a $2\times2$ downward row sum contraction of $M$ and $B=(b_{ij})$ be a $2\times2$ upward row sum contraction of $M$. Then
	$$\frac{\tr(A)}2+\sqrt{\left(\frac{\tr(A)}2\right)^2-\det(A)}
	\leq \rho(M)
	\leq\frac{\tr(B)}2+\sqrt{\left(\frac{\tr(B)}2\right)^2-\det(B)}$$
\end{corollary}
Again we call attention to the fact that we have required the diagonal blocks to be square submatrices as it is not true in general when the diagonal blocks are allowed to have arbitrary dimensions. In the gneral case, taking the minimum row sums of each block will not necessarily give a lower spectral radius. This is because we are relying on Theorem \ref{thm:row_contr_spec_rad_eq} to preserve the spectral radius when contracting the (downward and upward component-adjusted) matrices to a smaller size. 

Note that rather than performing contractions, one could perform expansions and then adjust elements up or down appropriately to get different bounds on the spectral radius, but generally calculating spectral radius for a larger matrix is more difficult or computationally intensive, so we don't state that result separately. 

Here is an example application of Corollary \ref{cor:rho2x2contr}. We give matrix $M$ and the bounds on its spectral radius from row sums alone.

$$
M=\left[\!
\begin{array}{ccc}
	1 & 3 & 2 \\
	5 & 1 & 1 \\
	2 & 4 & 3 
\end{array}\!\right] \quad 6\leq \rho(M)\leq 9
$$
Now we row-sum contract the upper left $2\times2$ block of $M$ to produce downward and upward contractions with spectral radii given below.
$$
\begin{aligned}
	M^\downarrow=\left[\!
	\begin{array}{ccc}
		4 & 1  \\
		6 & 3  
	\end{array}\!\right], 
M^\uparrow=\left[\!
	\begin{array}{ccc}
		6 & 2  \\
		6 & 3  
	\end{array}\!\right]
\Rightarrow \rho(M^\downarrow)=6\leq \rho(M)\leq 8.3\approx\rho(M^\uparrow)
\end{aligned}
$$
Thus our method gives a slight refinement over considering the row sums of $M$ alone. If we considered all possible $2\times2$ contractions of this matrix, the best bounds are approximately $6.3\approx2+\sqrt{19}\leq \rho(M)\leq (9+\sqrt{57})/2\approx8.3$. By considering column sums of $M$ we get $6\leq \rho(M)\leq 8$ and by considering all $2\times2$ column sum contractions we get approximately $6.5\approx(5+\sqrt{65})/2\leq \rho(M)\leq 8$, still a slight refinement over the basic method.

Here is an example where there is a slight refinement of the lower bound but a more significant refinement on the upper bound (relative to row or column sums alone). 
$$
M=\left[\!
\begin{array}{ccccc}
	\fbox1 & 3 & \fbox2 & 1 & \fbox2\\
	7 & \underline1 & 1 & \underline3 & 3\\
	\fbox2 & 4 & \fbox3 & 1 & \fbox0\\
	1 & \underline1 & 5 & \underline2 & 2\\
	\fbox4 & 3 & \fbox0 & 2 & \fbox1
\end{array}\!\right] \quad 8\leq \rho(M)\leq 15
$$
The best bounds are found with contracting the principle submatrices on indices $\{1,3,5\}$ (boxed above) and on $\{2,4\}$ (underlined above) each into a single component (i.e. using permutation $(1,2,3,4,5)\to(1,3,5,2,4)$ and partitioning into a $3\times3$ upper left diagonal block and $2\times 2$ lower right diagonal block) which results in the downward and upward row sum contraction matrices
$$
M^\downarrow=\left[\!
	\begin{array}{ccc}
		5 & 4  \\
		8 & 3  
	\end{array}\!\right],
M^\uparrow=\left[\!
	\begin{array}{ccc}
		5 & 5  \\
		11 & 4  
	\end{array}\!\right]
$$
giving the bounds
$$ 9.74\approx 4+\sqrt{33} =\rho(M^\downarrow)\leq \rho(M)\leq \rho(M^\uparrow)= \frac{9+\sqrt{153}}2\approx11.93$$
with the actual value being approximately $\rho(M)\approx10.995$.

Here is another straightforward corollary which discusses comparing the spectral radii of two matrices of arbitrarily dimensions.
\begin{corollary}
	Let $A,B$ be two square nonnegative matrices or arbitrary and possibly distinct dimension. If we can apply some sequence of permutations, transposes, and upward row or column sum contractions or expansions to $A$ and also some sequence of these operations to $B$ but with downward row or column sum contractions or expansions to create matrices satisfying $A^\uparrow\leq B^\downarrow$ (with both being of the same dimension) or more generally $\rho(A^\uparrow)\leq \rho(B^\downarrow)$ (with neither $A^\uparrow\not\leq B^\downarrow$ nor their being of the same dimension required), then $\rho(A)\leq\rho(B)$.
\end{corollary}
For example, consider square matrices of possibly distinct arbitrary dimensions written in block form with blocks of arbitrary dimensions (as long as the diagonal blocks are square submatrices) 
$A=\left[\begin{matrix}A_{11} & A_{12}\\A_{21}& A_{22}\end{matrix}\right]$ and $B=\left[\begin{matrix}B_{11} & B_{12}\\B_{21}& B_{22}\end{matrix}\right]$. Let $a_{ij}$ represent the maximum row sum of block $A_{ij}$, $A^\uparrow=(a_{ij})$, and $b_{ij}$ the minimum row sum of block $B_{ij}$, $B^\downarrow=(b_{ij})$. If we have that 
$$\frac{\tr(A^\uparrow)}2+\sqrt{\left(\frac{\tr(A^\uparrow)}2\right)^2-\det(A^\uparrow)}
\leq\frac{\tr(B^\downarrow)}2+\sqrt{\left(\frac{\tr(B^\downarrow)}2\right)^2-\det(B^\downarrow)}$$
then $\rho(A)\leq\rho(B)$.

Here is an example of comparing two matrices.

$$A=\left[\!
\begin{array}{c|cc|c}
	\fbox2 & 1 & 1 & \fbox2 \\\hline
	1 & 1 & 3 & 0 \\
	0 & 0 & 2 & 1 \\\hline
	\fbox1 & 2 & 0 & \fbox4 \\
\end{array}\!\right] \to 
 \tilde A=\left[\!
\begin{array}{cc|cc}
	1 & 3 & 1 & 0 \\
	0 & 2 & 0 & 1 \\\hline
	1 & 1 & \fbox2 & \fbox2 \\
	2 & 0 & \fbox1 & \fbox4 \\
\end{array}\!\right]
\to
A^\uparrow=\left[\!
\begin{array}{cc}
	4 & 1 \\
	2 & 5 \\
\end{array}\!\right]
$$
Note that we have applied a permutation to $A$ indicated by the boxed entries: $(1,2,3,4)\to(3,1,2,4)$.

$$B=\left[\!
\begin{array}{cc|c}
	1 & 2 & 2\\
	3 & 1 & 3 \\\hline
	1 & 1 & 5 \\
\end{array}\!\right]
\to
B^\downarrow=\left[\!
\begin{array}{cc}
	3 & 2 \\
	2 & 5 \\
\end{array}\!\right]
$$

In this case it is clear that $\rho(A^\uparrow)=6$ due to constant column sums and that  $\rho(B^\downarrow)=4+\sqrt{16-11}>6$ implying that $\rho(A)<\rho(B)$ even though there is no obvious relationship between their spectral radii by simply looking as row or column sums or comparing matrix components. Note that simply partitioning $A$ without any permutation
does not give the same spectral radius comparison and thus isn't useful.

\bibliographystyle{abbrv}



	
\end{document}